\documentclass[12pt, twoside, reqno]{amsart}
\usepackage{amsfonts}
\usepackage{amsmath}
\usepackage{ifthen}
\usepackage{amssymb}
\usepackage{color}
\usepackage{epsf,graphicx}

\numberwithin{equation}{section}

\newtheorem{theorem}{Theorem}[section]
\newtheorem{conjecture}[theorem]{Conjecture}
\newtheorem{corollary}[theorem]{Corollary}
\newtheorem{definition}[theorem]{Definition}

\newtheorem{lemma}[theorem]{Lemma}

\newtheorem{proposition}[theorem]{Proposition}

\newtheorem{quest}{Question}

\theoremstyle{definition}

\newcommand{\C}{{\mathbb C}}
\newcommand{\D}{{\mathbb D}}
\newcommand{\R}{{\mathbb R}}
\newcommand{\B}{{\mathbb B}}
\newcommand{\Hol}{{\operatorname{Hol}\,}}
\renewcommand{\Re}{{\operatorname{Re}\,}}

\newcommand{\N}{{\mathcal N}}

\newcommand {\Id}{\mathop{\rm Id}\nolimits}
\newcommand{\co}[1]{\overline{{\operatorname{conv}\,}#1}}

\begin{document}

\title[Non-linear resolvents]{Non-linear resolvents of holomorphically accretive mappings}
\date{\today}

\author[M. Elin]{Mark Elin}
\address{Braude College,
P.O. Box 78, Karmiel 21982, Israel}
\email{mark\_elin@braude.ac.il}


\maketitle

\begin{abstract}

In this paper, we present new results on holomorphically accretive mappings and their resolvents defined on the open unit ball of a complex Banach space. 

We employ a unified approach to examine various  properties of non-linear resolvents by applying a distortion theorem we have established. 
This method enables us to prove a covering result and to establish the accretivity of resolvents along with estimates of the squeezing ratio. 
Furthermore, we prove that under certain mild conditions, a non-linear resolvent is a starlike mapping of a specified order.

As a key tool, we first introduce a refined version of the inverse function theorem for mappings satisfying so-called one-sided estimates. 
\end{abstract}

\section{Introduction}\label{sect-intro}

The aim of this paper is to study on non-linear resolvents, which play a fundamental role in the theory of semigroups of holomorphic mappings.
These semigroups are a natural generalization of semigroups of linear operators.  While the theory of semigroups and their generators is widely developed in the one-dimensional case (see, for example, the recent books~\cite{B-C-DMbook, E-Sbook}),
in multi- and infinite-dimensional settings, the study of the  generation theory began with the works \cite{Ab-92} by Abate  and \cite{R-S-96} by Reich and Shoikhet. Over the past decades, various characterizations of semigroup generators have been established, some of them can be found in \cite[Theorem~7.5]{R-Sbook}, see also partial refinements in Proposition~\ref{prop-ineq} below).

\subsection{Object of study and questions}

One effective way to observe how certain properties of the generator affect the dynamic behavior of the generated semigroup is to use the so-called
product formula (formula~\eqref{expo-f} below). This formula involves the so-called non-linear resolvents $G_\lambda$ of the semigroup generator $f$.
These resolvents are defined by the formula $G_\lambda:=(\Id+\lambda f)^{-1}$, $\lambda>0$, see Definition~\ref{def-range_cond} below. 
Leaving aside their importance in dynamic systems, it is worth noting that non-linear resolvents form a class of biholomorphic self-mappings of the open unit ball, see~\cite{E-R-S-19, R-S-96, R-Sbook}, which is 
of inherent interest. 

Surprisingly, the study of geometric properties of non-linear resolvents has begun only recently in \cite{E-S-S}, where some important properties of resolvents in the open unit disk were first established. Partially these results were improved in \cite{E-J-resolv}. 

Naturally, obtaining multi-dimensional analogues of these results is a more complicated problem. One of the main reasons for this difficulty is that the null-point set of a holomorphic generator might be not an isolated point set but a manifold, see \cite[Chapter~8]{R-Sbook}. To date only partial generalizations of results in \cite{E-S-S} are known, see \cite{GHK2020, HK2022}. 

Specifically, in  this paper we deal with non-linear resolvents of semigroup generators  that are  holomorphic in the open unit ball $\B$ of a complex Banach space $X$ and  have an isolated zero at the origin. The prospect of removing this restriction will be discussed elsewhere. We develop the approach proposed in \cite{E-J-resolv} for the one-dimensional case and extend it to infinite-dimensional settings. \vspace{2mm}

The problems we study in this paper include:
\begin{quest}
Establish distortion and covering results for non-linear resolvents depending on the resolvent parameter $\lambda >0$.
\end{quest}
To the best of our knowledge, this issue has not yet been studied in multi-dimensional settings. In the one-dimensional case, this question was studied in \cite{E-J-resolv} using a global version of the inverse function theorem obtained earlier in \cite{E-S-20a}. 
Typically, the term `global inverse function theorem' refers to results concerning the existence of an inverse mapping in the whole space, exemplified by the well-known Hadamard (or Hadamard-L\'{e}vi) theorem, see, for instance, \cite{Ber, Ruz-Sug-2015}. However, in a holomorphic context,  it is more natural to examine mappings defined in a specific domain  and  wonder about the analytic extension of their inverses. 
Therefore, even before exploring Question~1, we plan to study the next problem: 
\begin{quest}
  Develop a satisfactory version of the inverse function theorem applicable to non-linear resolvents.
\end{quest}

To proceed, recall that the set of  starlike mappings is a very attractive object of study in geometric function theory (
the reader is referred to \cite{E-R-S-04, GK2003, HKL, STJ-77}, see also Definition~\ref{def-star} below). In the one-dimensional case it was proven in \cite{E-S-S} that any resolvent is a starlike function of order at least $\frac12\,.$ For more precise estimates of order of starlikeness see~\cite{E-J-resolv}. As for multi-dimensional settings, starlikeness of resolvents for small values of the parameter $\lambda$ was obtained in \cite{GHK2020} under an additional restriction, see  Proposition~\ref{propo-star-GHK} below. Thus, it is natural to raise the problem:
\begin{quest}
Find the sharp order of starlikeness for non-linear resolvents  depending on the resolvent parameter .
\end{quest}
Moreover, following the analogy with the one-dimensional case,  we suggest that the following conjecture is valid. \begin{conjecture}\label{conje-1}
Let $\left\{G_\lambda\right\}_{\lambda>0}$ be the resolvent family of a holomorphically accretive mapping $f$. Then every mapping $G_\lambda$ is starlike of order at least $\frac12$. Furthermore, the order of starlikeness of $G_\lambda$ tends to~$1$ as $\lambda\to\infty$.
\end{conjecture}

\subsection{Outline and main results}
Since the main objects of our study are non-linear resolvents of holomorphically accretive mappings, in the next Section~\ref{sect-prelim}, we provide the necessary notation, define holomorphically accretive  mappings and non-linear resolvents, formulate their basic properties. We also prove those of them that do not appear in the literature in the form we need.

As previously mentioned, the declared approach to the study of non-linear resolvents relies on a qualified version of the inverse function theorem.  In Section~\ref{sect-invers-fun}, we establish a global holomorphic version of this theorem. 
Although our result  (Theorem~\ref{thm-main1})  is of independent interest, its formal statement is quite long and complicated, prompting us to defer it to Section~\ref{sect-invers-fun}.

Next, in Section~\ref{sect-resolv-1}, using the inverse function theorem, we prove distortion and covering results for non-linear resolvents.  

In what follows we assume that $a\ge0$, $f$ is a semigroup generator of the class $\N_a$ (see Definition~\ref{def-dissip} below). Let $K_A$ be the support function of the numerical range $V$ of the linear operator $A=f'(0)$, that is, $K_A(\theta):=\sup_{w\in V} \Re e^{-i\theta}w$. Denote $M_f(r):=\sup_{\|y\|<r} \|f(y)\|$, $K:=K_A(0)$, $\alpha(\lambda):=\min\left(  \frac{3}{1-\lambda K} , \frac1{1+\lambda a} \right)$ and $\beta(\lambda):= \frac{\alpha(\lambda)}{\alpha(\lambda)+\lambda M_f(\alpha(\lambda))}\,.$

\begin{theorem}\label{thm-dist-cover}
   Let $G_\lambda$ be the resolvent of $f$ corresponding to the parameter $\lambda>0$. Then for any $x\in\B$  we have
   \[
   \beta(\lambda)\cdot \|x\| \le\left\|  G_\lambda(x)  \right\|  \le \alpha(\lambda)\cdot \|x\|.
   \]
\end{theorem}

As an immediate consequence, we conclude that the resolvent family $\{G_\lambda\}_{\lambda>0}$ converges to zero as $\lambda\to\infty$, uniformly on $\B$, and that if $X$ is finite-dimensional, the image $G_\lambda(\B)$ covers the ball of radius~$\beta(\lambda)$.

In the last Section~\ref{sect-starlike}, we consider starlikeness properties of resolvents. Namely, we apply Theorem~\ref{thm-dist-cover} to partially prove Conjecture~\ref{conje-1}. Using the auxiliary function $\Psi$, $  \Psi(\lambda) :=  \frac{2\lambda\left( K_A(0) -a  \right) }{1-\lambda K_A(\pi)} \cdot  \sum\limits_{n=2}^\infty   n^{\frac{2n-1}{n-1}} (\alpha(\lambda) )^{n-1} $, we prove the following.
  \begin{theorem}\label{thm-order1}
  Let  $\{G_\lambda\}_{\lambda>0}$ be the resolvent family for $f$. Denote by $\lambda^*$ the largest solution to the equation ${\Psi(\lambda)= 1}.$ Then for every $\lambda\ge\lambda^*$, the mapping $G_\lambda$ is starlike of order~$\frac12$.
\end{theorem}

Note in this regard that in \cite{GHK2020} it was proven for finite-dimensional case that if $f'$ is bounded in $\B$ and $f'(0)=\Id$, then for small values of the parameter $\lambda $, the mapping $G_\lambda$ is starlike (see Proposition~\ref{propo-star-GHK} below).  In Theorem~\ref{thm-order2} we refuse the normalization and establish order of starlikeness for both small and large values of $\lambda$. 

\medskip


\section{Preliminaries}\label{sect-prelim}

\setcounter{equation}{0}

Let $X$ be a complex Banach  space equipped with the norm $\|\cdot\|$, $\B_r:=\left\{x\in X:\ \|x\|<r \right\}$, and $\B:= \B_1$ be the open unit ball in $X$. We denote by $X^*$  the space of all bounded linear functionals on $X$ (the dual space to $X$)  with the duality pairing $\langle x,a\rangle,\ x\in X, a\in X^*$. For each $x\in X$, the set $J(x)$, defined~by
\begin{equation}\label{Jx-set}
J(x):=\left\{ x^{\ast }\in X^{\ast }:\ \left\langle x,x^{\ast }
\right\rangle =\left\Vert x\right\Vert ^{2}= \left\Vert x^{\ast
}\right\Vert ^{2}\right\},
\end{equation}%
is non-empty according to the Hahn--Banach theorem (see, for example, \cite[Theorem~3.2]{Rudin}). It may consists of a singleton (for instance, when $X$ is a Hilbert space), or, otherwise, of infinitely many elements. Its elements $x^* \in J(x)$ are called support functionals at the point $x$. 

Let $Y$ be another complex Banach space and $D\subset X,\ D_1\subset Y$ be domains. A mapping $f:D\to D_1$ is called holomorphic if it is Fr\'echet differentiable at each point $x\in D$ (see, for example, \cite{E-R-S-19, GK2003, R-Sbook}). The set of all holomorphic mappings from $D$ into $D_1$ is denoted by $\Hol(D, D_1)$ and $\Hol(D)=\Hol(D,D)$. 

\vspace{2mm}

The notion of numerical range for linear operators acting on a complex Banach space was introduced by Lumer \cite{Lum}. We recall it in the framework of the space $L(X)$ of bounded linear operators on $X$. The numerical range of an operator $A\in L(X)$ is the~set 
\begin{equation}\label{4.2}
V(A):=\left\{ \left\langle Ax,x^{\ast }\right\rangle ,\ x\in\partial\B,\  x^{\ast }\in J(x)\right\}\subset\C.
\end{equation}%

In what follows we denote by $K_A(\theta)$ the support function of the set $V(A)$, that is,
\[
K_A(\theta):=\sup_{w\in V(A)} \Re e^{-i\theta}w.
\]

The operator  $A$ is dissipative (respectively, strongly dissipative, accretive, strongly accretive) if and only if $K_A(0)\le0$ (respectively, $K_A(0)<0$, $K_A(\pi)\le0$, $K_A(\pi)<0$).

\vspace{2mm}

This concept was extended to non-linear holomorphic mappings by Harris \cite{Har}. For the case of a holomorphic mapping $h$ on the open unit ball $\B$ having a uniformly continuous extension to the closed ball $\overline{\B}$, the numerical range of $h$ is defined by 
\[
V(h):=\left\{\left\langle h(x),x^*\right\rangle:\ x\in\partial\B,\  x^*\in J(x)  \right\}.
\]

\begin{proposition}[see \cite{Har}]\label{prop-Har1}
Let $h\in\Hol(\B,X)$, $h(0)=0$  and $A=h'(0)$. Then $\co{V(h)}=\overline{ \left\{\left\langle h(x),x^*\right\rangle:\ x\in\B,\  x^*\in J(x)  \right\}}$ and $V(A)\subset \co{V(h)} $. Consequently, for any $\theta\in\R$,
\begin{equation}\label{K-dop}
K_A(\theta) \le  \sup_{x\in\B,\ x^*\in J(x)}\Re\langle e^{-i\theta}h(x),x^*\rangle .
\end{equation}  
\end{proposition}

The numerical range finds numerous applications in multi- and infinite-dimensional analysis, geometry of Banach spaces and other fields (see, for example, the recent books~\cite{GK2003, R-Sbook, E-R-S-19}). 
It also serves as the base to the study of so-called holomorphically accretive/dissipative mappings (not nessecarily continuous on 
$\overline \B$). Namely,

$\diamond$ {\it for $f\in\Hol(\B,X)$ let denote following \cite{Har, H-R-S},
\begin{equation}\label{n}
n(f):=\liminf_{s\to1^-} \inf\left\{\Re w:\ w\in V(f(s\,\cdot)) \right\}.
\end{equation}
Then $f$ is said to be holomorphically accretive if $n(f)\ge0$. The mapping $f$ is called holomorphically dissipative if $-f$ is holomorphically 
accretive
.} 
Thus, $f\in\Hol(\B,X),\ {f(0)=0},$ is holomorphically accretive if and only if $\Re \left\langle f(x),x^*\right\rangle\ge0$ for all $x\in\B$.  

The next assertion is an immediate consequence of Theorems~6.11 and~7.5 in \cite{R-Sbook}. We say that $f\in\Hol(\B,X)$ is an {\it infinitesimal generator} if for any $x\in\B$ the Cauchy problem
\begin{equation}  \label{semig}
\left\{
\begin{array}{l}
\displaystyle
\frac{\partial u(t,x)}{\partial t} +f(u(t,x)) =0, \vspace{2mm} \\
u(0,x)=x,%
\end{array}%
\right.
\end{equation}
has the unique solution $u(t,x)\in\B$ for all $t\ge0$.

\begin{proposition}[see \cite{Har}]\label{prop-RS1}
Let $f\in\Hol(\B,X)$, $f(0)=0$,  and $f$ is bounded on each subset strictly inside $\B$. Then $f$ is an infinitesimal generator if and only if it is holomorphically accretive.
\end{proposition}

Holomorphically accretive/dissipative mappings play an essential role not only for dynamical systems, but also in geometric function theory (the  reader is referred to \cite{E-R-S-19, GK2003, H-R-S}). The set of all holomorphically accretive mappings vanishing at the origin is often denoted by $\N_0$ (see, for example,  \cite{GK2003}). Hence the following definition and notation are natural.

\begin{definition}\label{def-dissip}
A mapping $f\in\Hol(\B,X),\ f(0)=0,$ is said to be holomorphically accretive if
\begin{equation}\label{a-condi}
  \Re\langle f(x),x^*  \rangle \ge a\|x\|^2 \quad \mbox{ for all }\quad  x\in\B \quad \mbox{ and } \quad  x^*\in J(x)
\end{equation}
 for some $a\ge0$. We denote the class of mappings that satisfy \eqref{a-condi} with a given $a\ge0$ by $\N_a$. 
 %
 \end{definition}
 Note that for any  holomorphically accretive 
 mapping~$f$, the number~$a$ can be chosen to be zero. Whence $a>0$ the mapping $f$ is named {\it strongly holomorphically accretive
 }. The interest in such mappings is based  on the following statement.

 \begin{proposition}\label{propo-sque}
   Let $f$ be a holomorphically accretive mapping on $\B$ and $\left\{u(t,\cdot) \right\}_{t\ge0}$ be a semigroup generated by $(-f)$.
   \begin{itemize}
     \item  [(i)] If $f\in\N_a,\ a>0,$ then $\left\{u(t,\cdot) \right\}_{t\ge0}$ converges to zero uniformly on $\B$ with the squeezing ratio $\kappa=a$, that is,
    \begin{equation}\label{squee}
      \left\| u(t,x) \right\| \le e^{-at} \|x\|\quad \mbox{for all }\ x\in\B.
    \end{equation}
     \item [(ii)] If $X$ is a Hilbert space and the semigroup $\left\{u(t,\cdot) \right\}_{t\ge0}$ satisfies inequality \eqref{squee} for some $a>0$, then $f\in\N_a$.
   \end{itemize}
 \end{proposition}
 Assertion (i) follows from \cite[Lemma~3.3.2]{E-R-S-04}, where we set $\alpha(s)=as$, while assertion (ii) is proved in \cite{E-R-S-02}, see also \cite{R-Sbook, E-R-S-19}. If $X$ is not necessarily a Hilbert space but the function $t\mapsto   \left\| u(t,x) \right\|^2$ is still differentiable for every $x\in\B$, it can be proven that $f\in\N_a$ by repeating the proof from \cite{E-R-S-02} and using \cite[Lemma~1.3]{Kato67}.\vspace{2mm}

\noindent{\bf Assumption A:} for $f\in\N_a$, we assume that $f'(0)$ is a strongly accretive operator, hence $K_{f'(0)}(\pi)<0$ and $K_{f'(0)}(\pi)+a\le0$. 

The next assertion is a certain generalization to the classes $\N_a$ of the well-known characterization of the class $\N_0$  (the equivalence of  assertions (i), (iii) and (iv) below  with $a=0$ is the content of \cite[Theorem~7.5]{R-Sbook}, see also reference therein).

  \begin{proposition} \label{prop-ineq}
   Let $a\ge0$ and $f\in\Hol(\B,X),\ f(0)=0.$ Denote $A:=f'(0)$. The following assertions are equivalent:
    \begin{itemize}
      \item [(i)] $f\in\N_a$;

      \item [(ii)] for every $x\in\B$ the value of the function $\left\langle f(x),x^* \right\rangle$ lies in the closed disk centered at the point 
       $$ c(x):= \frac1{1-\|x\|^2} \left( \langle Ax,x^* \rangle +\|x\|^2 \overline{\langle Ax,x^* \rangle} - 2a \|x\|^4  \right) $$ 
       and of radius $\displaystyle r(x):=\frac{2\|x\|}{1-\|x\|^2}\left( \Re\langle Ax,x^*\rangle -a\|x\|^2 \right)$;

      \item [(iii)] for every $x\in\B$ the following inequality holds
   \begin{eqnarray}%
   \Re \left\langle Ax,x^*\right\rangle \frac{1-\|x\|}{1+\|x\|} +\frac{2a\|x\|^3}{1+\|x\|} \le \Re \left\langle f(x),x^*\right\rangle \label{ineq-aux1}  ;
   \end{eqnarray}
      \item [(iv)] for every $x\in\B$ the following inequality holds
   \begin{eqnarray}%
   \Re\left[ 2  \left\langle f(x),x^*\right\rangle +(1-\|x\|^2) \left\langle f'(x)x,x^*\right\rangle \right] \ge  a \|x\|^2 \left( 1+ \|x\|^2 \right) .     \label{ineq-aux2}
   \end{eqnarray}
    \end{itemize}
 \end{proposition}
 \begin{proof}
    Let  $f\in\N_a$. Take any point $x_0\in\B$ and write it in the form $x_0=z_0u_0$, where $u_0\in\partial\B$ and $z_0\in\C$ with $|z_0|<1$.     Consider the function $q$ defined as follows:
    \begin{equation}\label{q}
    q(z):=\left\{\begin{array}{cc}
             \displaystyle    \frac1z \left\langle f(zu_0),u_0^*\right\rangle -a, & z\neq0,\vspace{2mm} \\
                    \left\langle A u_0,u_0^*\right\rangle-a, & z=0.
                 \end{array}
   \right.
    \end{equation}
It is analytic in the open unit disk $\D$ and has non-negative real part because
\[
\Re q(z) =\frac1{|z|^2}\Re \left\langle f(zu_0),(zu_0)^*\right\rangle -a\ge 0 \quad\mbox{for}\quad z\in\D\setminus\{0\}.
\]
Therefore
\[
\left| q(z)-\frac{q(0)+|z|^2 \overline{q(0)}}{1-|z|^2}   \right| \le \frac{2|z|}{1-|z|^2}\,\Re q(0).
\]
According to \eqref{q}, this coincides with
\[
\left| \left\langle f(zu_0),(zu_0)^*\right\rangle -a|z|^2 -\frac{|z|^2q(0)+|z|^4 \overline{q(0)}}{1-|z|^2}   \right| \le \frac{2|z|^3}{1-|z|^2}\,\Re q(0).
\]
The last inequality implies assertion (ii).

If assertion (ii) holds, then
\[
  \Re c(x) -r(x) \le \Re \left\langle f(x),x^*\right\rangle,
\]
which coincides with assertion (iii).

Assume now that assertion (iii) holds, take as above $x_0=z_0u_0\in\B$ and consider again the function $q$ defined by \eqref{q}. It follows from \eqref{ineq-aux1} that
\[
\Re q(z) =\frac1{|z|^2}\Re \left\langle f(zu_0),(zu_0)^*\right\rangle-a \ge \left( \Re \left\langle Au_0,u_0^* \right\rangle -a \right) \frac{1-|z|}{1+|z|} \,.
\]
The using of the maximum principle leads to $$\Re q(z) \ge \lim_{r\to1^-} \left( \Re \left\langle Au_0,u_0^* \right\rangle -a \right) \frac{1-r}{1+r} =0.$$
In particular, $\Re q(z_0)= \frac1{\|x_0\|^2}\Re \left\langle f(x_0),x_0^*\right\rangle-a \ge 0 $. Since the point $x_0\in\B$ is arbitrary, this means that $f\in\N_a$, that is, assertion (i) holds. Thus we have proved that assertions (i), (ii) and (iii) are equivalent.

On the other hand, consider  the mapping $g$ defined by $g(x)=f(x)-ax$.  Obviously, it satisfies $g(0)=0$ and $g'(0)=A-a\Id$.  Assertion (i) means that $g$ is holomorphically accretive, while inequality \eqref{ineq-aux2} gets the form:
 \begin{eqnarray*}
     \Re\left[ 2 \|x\|^2 \left\langle g(x),x^*\right\rangle +(1-\|x\|^2) \left\langle g'(x)x,x^*\right\rangle \right] \ge 0.
    \end{eqnarray*}
    Hence the equivalence of assertions (i) and (iv) follows from  \cite[Theorem~7.5]{R-Sbook}.
   \end{proof}

 The next result is a version of Theorem~3.2 in \cite{G-H-H-K-S} for non-normalized holomorphically accretive mappings with a slightly modified proof.
 \begin{proposition}\label{propo-esim1}
   Let $f\in\N_a$ be represented by the series of homogeneous polynomials: $f(x)=Ax +\sum
   _{n=2}^\infty P_n(x)$.
      Then $\left\| P_n \right\| \le  2 n^{\frac{n}{n-1}}\left( K_A(0) -a  \right)$,  $n\ge2$, where $\left\| P_n \right\| :=\sup_{\|x\|<1}\|P_n(x)\|$.
    \end{proposition}
    \begin{proof}
     Fix $u_0\in\partial\B$ and define the  holomorphic function $q$ with $\Re q(z)\ge0,\ z\in\D,$ by formula \eqref{q} above. Consider the case $\Re q(0) >0$. Then the point $\tau
     = \frac{q(0)-1}{q(0)+1}$ belongs to the open unit disk. Denote by $m$ the M\"obius involution $m(z)=\frac{\tau-z}{1-\overline{\tau}z}$ and define the function $g$ by $g(z)=\frac{1+m(z)}{1-m(z)}$. Obviously, $|g'(0)|=2\Re q(0)$ and $g(0)=q(0)$. This function $g$ is a conformal mapping of the open unit disk onto the right half-plane. Hence $q\prec g$.
     By a result of Rogosinski \cite{Rogo-43}, this implies that the Taylor coefficients $q_n$ of $q$ satisfy $|q_n|\le |g'(0)|=2\Re q(0) \le 2\left( K_A(0)-a  \right)$.

Further, formula \eqref{q} implies that $q_n=\left\langle P_{n+1}(u_0),u_0^*\right\rangle,\ n\ge1$. Since $u_0\in\partial\B$ and $u_0^*\in J(u_0)$ are arbitrary, one concludes that $|V(P_{n+1})|\le 2\left( K_A(0) -a \right)$. Using \cite[Theorem~1]{Har}, one completes the proof.
    \end{proof}
\vspace{2mm}
Now we define non-linear resolvents, the main object of the study in this paper.

\begin{definition}\label{def-range_cond}
Let $h\in \Hol(\B,X)$. One says that $h$ satisfies the range condition on $\B$ if $\left( \Id+\lambda h\right) (\B) \supset\B$
  for each $\lambda > 0$  and the so-called resolvent equation
\begin{equation}\label{G*}
 w+ \lambda h(w)=x
\end{equation}
has a unique solution
\begin{equation} \label{resolvent1}
w=G_{\lambda }(x)\left( =\left( \Id + \lambda h\right) ^{-1} (x) \right)
\end{equation}%
holomorphic in $\B.$ If it is the case, every mapping $G_\lambda,\ \lambda>0,$ is  called the non-linear resolvent and the family $\{G_{\lambda }\}_{\lambda \geq 0}\in \Hol(\B)$ is called the resolvent family of $h$ on $\B.$
\end{definition}

The following fact appeared at first in \cite{R-S-96} (see also the recent books \cite{R-Sbook, E-R-S-19}).
\begin{proposition}\label{lemma1}
  A mapping $h\in \Hol(\B,X)$ is holomorphically accretive if and only if it satisfies the range condition on $\B$.
\end{proposition}

Numerous features of non-linear resolvents can be found in the books \cite{R-Sbook, E-R-S-19}. In particular, the solution of the Cauchy problem~\eqref{semig} can be reproduced by the following exponential formula:
\begin{equation}\label{expo-f}
u(t,\cdot) = \lim_{n\rightarrow \infty } \left(G_{\frac{t}{n}}\right)^{[n]},
\end{equation}
where $G^{[n]}$ denotes the $n$-th iterate of a self-mapping $G$ and the limit exists in the topology of uniform convergence on compact subsets of $\D.$

Recently non-linear resolvents on the finite-dimensional Euclidean ball were studied in \cite{GHK2020}. In particular, the quasiconformality of resolvents was established. In addition, it was proved that the resolvent family $\{G_\lambda\}_{\lambda\ge0}$ is an inverse L{\oe}wner chain on $\B$. The last result was extended to the case of Hilbert spaces in \cite{HK2022}.

\medskip

\section{Inverse function theorem}\label{sect-invers-fun}

\setcounter{equation}{0}

In this section we present a version of the inverse function theorem  for mappings holomorphic in the open unit ball $\B$.  
We use so-called one-side estimates that, in fact, localize the numerical range of a mapping in some half-plane. It seems that for a holomorphic inverse function theorem this method was first used in \cite{H-R-S} (see also \cite{Har}).

\begin{theorem}\label{thm-main1}
  Let $h\in\Hol(\B,X)$ satisfy $h(0)=0$ and let $A:=h'(0)$. Assume that 
   for some $\theta\in\R$ and $c\ge0$ 
   the following inequality holds:
  \begin{equation}\label{ineq}
      \Re \langle e^{-i\theta} h(x),x^*\rangle \le - c\|x\|^2 \quad \mbox{for all}\quad x\in\B,\ x^*\in J(x).
  \end{equation}
  Then the inverse mapping $h^{-1}$ is  holomorphic in the ball $\B_\rho$, where
  \[
 \rho:= \left\{ \begin{array}{cl}
                 c, & \  \mbox{as } K_A(\theta)>-3c,\vspace{2mm}\\
                \left( \sqrt{|2c+ 2K_A(\theta)|} - \sqrt{|2c+K_A(\theta)|} \right)^2 \,, &  \ \mbox{as }  K_A(\theta)\le-3c,
                \end{array}
 \right.
  \]
  and maps it biholomorphically into the ball $\B_R$, where
  \[
   R := \left\{ \begin{array}{cl}
                1, & \  \mbox{as } K_A(\theta)>-3c,\vspace{2mm}\\
                \sqrt{\frac{2c + 2K_A(\theta)}{2c +K_A(\theta)}}-1, &  \ \mbox{as }  K_A(\theta)\le-3 c.
                \end{array}
 \right.
  \]
\end{theorem}

\begin{proof} 
Let us apply Proposition~\ref{prop-ineq} to the mapping $e^{-i(\theta+\pi)}h$. Then denoting $r=\|x\|$ we have
  \begin{eqnarray*}
   \Re \left\langle e^{-i\theta}h(x),x^*\right\rangle &\le& \Re \left\langle e^{-i\theta}Ax,x^*\right\rangle \frac{1- r}{1+ r} - \frac{2c r^3}{1+ r}  \\
   &\le& \frac{r^2}{1+r} \left[ (1-r)K_A(\theta) - 2cr \right].
  \end{eqnarray*}

Hence for every real $t$ one concludes
\begin{eqnarray}\label{first_ineq}
   &&\Re \langle e^{-i\theta} h(x)+tx, x^* \rangle \le  \frac{r^2}{1+r} \left[ (1-r)K_A(\theta) -2rc \right] +r^2t \nonumber \\
  && \hspace{2.5cm} = r^2\left[ K_A(\theta) +t - \frac{2r}{1+r}\left( c+ K_A(\theta) \right) \right].
\end{eqnarray}

Consider the sign of the right-hand term in \eqref{first_ineq} for a fixed value of $t$ in the following three cases:
\begin{itemize}
  \item [(a)]  $t+K_A(\theta)>0$,
  \item [(b)]  $t-c\le0$ and
  \item [(c)]  $c< t\le -K_A(\theta) $.
\end{itemize}

In case (a),  $K_A(\theta) +t - \frac{2r}{1+r}\left( c+ K_A(\theta) \right)>0$ for all $r<1$. In case (b), this expression is non-positive  for all $r<1$. So, we concentrate on case (c). In this case $K_A(\theta) +t - \frac{2r}{1+r}\left( c + K_A(\theta) \right)\le0$ if and only if
\[
r\left(-2c+t- K_A(\theta) \right)\le -(t+K_A(\theta)).
\]
Since ${K_A(\theta)\le -c}$ by \cite[Proposition~2.3.2]{E-R-S-19}, the last displayed inequality is equivalent to $\displaystyle r\le\frac{-(t+K_A(\theta))}{-2c +t- K_A(\theta)}$\,. Hence it is natural to denote
\begin{equation}\label{r}
  r(t):=\left\{
  \begin{array}{cl}
    1, & t\le c  \vspace{2mm}\\
\displaystyle   \frac{t+K_A(\theta)}{2c - t + K_A(\theta)}\,, & c< t\le -K_A(\theta) .
  \end{array}
  \right.
\end{equation}
In fact, our previous calculation together with inequality \eqref{first_ineq} show that the mapping $e^{-i\theta}h+t\Id $ \ is holomorphically dissipative on the ball of radius $r(t)$. Then the mapping $h_1$ defined by
\begin{equation}\label{h1}
  h_1(x):= -\left( e^{-i\theta}h(r(t)x) + r(t) t x \right),
\end{equation}
is holomorphically accretive on $\B$. Therefore, Proposition~\ref{lemma1} implies that for every $\lambda>0$ and every $\widehat{y}\in\B$ the functional equation
\[
\widehat{x} + \lambda h_1(\widehat{x}) =\widehat{y}
\]
has a unique solution $\widehat{x}=\widehat{x}(\lambda,\widehat{y})\in\B$. Moreover, $\widehat{x}(\lambda,\cdot)\in\Hol(\B)$. Taking in mind \eqref{h1}, we obtain that $\widehat{x}\in\B$ solves the functional equation
\[
\widehat{x}-\lambda \left[ e^{-i\theta}h(r(t)\widehat{x}) +tr(t)\widehat{x} \right] =\widehat{y}.
\]
In particular, for $t>0$ and $\lambda=\frac1{tr(t)}$ the last equation takes the form
\[
h(r(t)\widehat{x}) = - e^{i\theta}tr(t)\widehat{y}.
\]
Denoting $x=r(t)\widehat{x}$ and $y=- e^{i\theta}tr(t)\widehat{y}$, we conclude that for every $y\in\B_{tr(t)}$ there is $x=x(y)\in\B_{r(t)}$ such that $h(x)=y$. Moreover, $x=x(y)$ is a holomorphic mapping.

Summarizing, for every $t\in(0,-K_A(\theta))$, the inverse mapping $h^{-1}$ is a well-defined holomorphic mapping in the ball $\B_{\rho(t)}$ of radius
\begin{equation}\label{rho}
    \rho(t):=\left\{
  \begin{array}{cl}
    t, & 0<t\le c   \vspace{2mm}\\
\displaystyle   \frac{t(t+K_A(\theta))}{2c - t + K_A(\theta)}\,, & c< t\le -K_A(\theta) .
  \end{array}
  \right.
\end{equation}
and takes values in the ball of radius $r(t)$ defined by \eqref{r}.

Let maximize the function $\rho(t)$ on $(0,-K_A(\theta)]$. To this end, write $K_A(\theta)=-cb$ with some $b\ge1$ and consider the function $\rho_1$ defined by $\rho_1(\tau)=\frac{1}{c}\rho(c\tau),\ \tau\in[1,b]$. Explicitly we have $\rho_1(\tau)=\displaystyle\frac{\tau^2-b\tau}{2-b-\tau}$.

It is directly verified that if $b<2$, then $\rho_1$ is decreasing on $[1,b]$. Hence its maximum is attained at $\tau=1$ with $\rho_1(1)=1$.

If $b\ge 2$, the critical points of $\rho_1$ are $\displaystyle \tau_1=2-b-\sqrt{2(b-1)(b-2)}$  and $ \tau_2=2-b+\sqrt{2(b-1)(b-2)}.$

A straightforward calculation shows that $\tau_1<1$ and $\tau_2<b$.  Moreover, if $b\in[2,3]$, then $\tau_2\le1$. Therefore $\rho_1$ is a decreasing function on $[1,b]$ with the same maximum value $\rho_1(1)=1$.

Otherwise, when $b>3$, we have $\tau_2\in(1,b)$ is the maximum point of $\rho_1$. Now we just calculate:
\[
\rho_1(\tau_2)=3b-4-2\sqrt{2(b-1)(b-2)}.
\]

Since $b=\frac{-K_A(\theta)}{c}$, $\rho(t)= c \rho_1(t/c)$ and $r(t)=\frac1t\rho(t)=\frac{c}{t}\rho_1(t/c)$, the result follows.
\end{proof}

In the case where $h$ is a normalized mapping, that is, $h'(0)=\Id$, this theorem is reduced to the following:
\begin{corollary}\label{cor-HRS} 
  Let $h\in\Hol(\B,X)$ satisfy $h(0)=0$ and $h'(0)=\Id$. Assume that $e^{-i\theta}h\in\N_c$ for some $\theta\in\left(-\frac\pi2,\frac{\pi}2\right)$ and $c\in[0, \cos\theta)$.  Then the inverse mapping $h^{-1}$ is  holomorphic in the ball $\B_\rho$, where
  \[
 \rho:= \left\{ \begin{array}{cl}
                 c, & \  \mbox{as } 3c- \cos \theta >0,\vspace{2mm}\\
                \left( \sqrt{2|c + \cos\theta|} - \sqrt{|2c+\cos\theta|} \right)^2 \,, &  \ \mbox{as }  3c- \cos \theta \le0,
                \end{array}
 \right.
  \]
  and maps it biholomorphically into the ball $\B_R$, where
  \[
   R := \left\{ \begin{array}{cl}
                1, & \  \mbox{as } 3c- \cos \theta >0,\vspace{2mm}\\
                \sqrt{\frac{2(c +\cos\theta)}{2c +\cos\theta}}-1, &  \ \mbox{as } 3c- \cos \theta \le0.
                \end{array}
 \right.
  \]
  \end{corollary}

  Choosing in this corollary $\theta=0$, we get Theorem~7 in \cite{H-R-S}. The one-dimensional case of Theorem~\ref{thm-main1} was investigated in \cite{E-S-20a}.

\medskip

\section{Distortion and covering properties of resolvents}\label{sect-resolv-1}

\setcounter{equation}{0}

The purpose of this section is to establish covering and distortion theorems for families of non-linear resolvents holomorphically accretive mappings of the class $\N_a,\ a\ge0,$ that is, such mappings $f$ that satisfy $f(0)=0$ and 
\begin{equation}\label{a-condi-1}
  \Re\langle f(x),x^*  \rangle \ge a\|x\|^2 \qquad \mbox{ for all }\  x\in\B \ \mbox{ and } \  x^*\in J(x).
\end{equation}
We also assume that Assumption A holds, that is, $f'(0)$ is a strongly accretive operator, hence $K:=K_{f'(0)}(\pi)<0$ and $K+a\le0$.

\begin{theorem}\label{thm-reso-general}
  Let $f\in\N_a$ for some $a\ge0$ and $G_\lambda\in \Hol(\B)$ be the resolvent of $f$ corresponding to the parameter $\lambda>0$. If $3a+K<0$ and $\lambda> \frac{2}{|3a+K|}$\,, the mapping $G_\lambda$ admits biholomorphic extension onto the ball $\B_{\rho(\lambda)}$, where
  \[
  \rho(\lambda)=\left( \sqrt{-2\lambda(a+K)} - \sqrt{-(2\lambda a +\lambda K+1)}   \right)^2 ,
  \]
  and maps it into the ball $\B_{R(\lambda)}$, where $\displaystyle R(\lambda)=\sqrt{\frac{2\lambda (a+K)}{2\lambda a+\lambda K+1}}-1.$

  Otherwise,  if either $3a+K\ge0$ or $\lambda\le \frac{2}{|3a+K|}$, the mapping  $G_\lambda$ admits biholomorphic extension onto the ball $\B_{1+\lambda a}$ and maps it into~$\B$.
  \end{theorem}
 Notice that the case $3a+K<0$ is generic since $K<0$ and one can always choose $a=0$.

\begin{proof}
Denote $h(x):= - \lambda f(x)-x$. Then the resolvent equation gets the form $h(x)=-y,\ y\in\B,$ and the mapping $G_\lambda$ is the inverse to~$(-h)$. Note that $A:=h'(0)=-\lambda f'(0)-\Id$, so that $K_A(0)=\lambda K -1$. Also we have
\[
\Re \langle  h(x),x^*\rangle =- \lambda\Re  \langle f(x),x^*  \rangle - \|x\|^2\le - \left(\lambda a+1\right)\|x\|^2.
\]
Thus $h$ satisfies inequality \eqref{ineq} in Theorem~\ref{thm-main1}  with $\theta=0$ and $c=\lambda a+1$. The condition $K_A(0)< -3c$ from that theorem holds if and only  if $3a+K<0$ and $\lambda> \frac{2}{|3a+K|}$. Substituting $K_A(0),\ \theta$ and $c$, we conclude by Theorem~\ref{thm-main1} that  $G_\lambda$ maps the ball $\B_{\rho(\lambda)}$ into the ball~$\B_{R(\lambda)}$.

Otherwise,  if either $3a+K\ge0$ or $\lambda\le \frac{2}{|3a+K|}$, we have $K_A(0)\ge -3c$. Substitute again $K_A(0),\ \theta$ and $c$ in the statement of Theorem~\ref{thm-main1} and obtain that $G_\lambda$ admits biholomorphic extension onto the ball $\B_{1+\lambda a}$ and maps it into~$\B$.
\end{proof}

  \begin{corollary}\label{cor-unif}
    Denote $\alpha(\lambda):=\min\left(  \frac{3}{1-\lambda K} , \frac1{1+\lambda a} \right)$. Then for all $\lambda>0$ and $x\in\B$  we have
\begin{equation}\label{ineq-reso1}
   \left\|  G_\lambda(x)  \right\|  \le \alpha(\lambda)\cdot \|x\| ,
\end{equation}
so the resolvent family $\{G_\lambda\}_{\lambda>0}$ converges to zero as $\lambda\to\infty$, uniformly on $\B$.
  \end{corollary}
    \begin{proof}
  First assume that $3a+K<0$ and $\lambda> \frac{2}{|3a+K|}$. Since $G_\lambda(0)=0$, the above theorem and the Schwarz lemma imply that $ \left\| G_\lambda(x) \right\|\le \frac{R(\lambda) \|x\|} {\rho(\lambda)}\,.$ Thus
\[
   \left\|  G_\lambda(x)  \right\|  \le \frac{\sqrt{\frac{2\lambda a+2\lambda K}{2\lambda a+\lambda K+1}}+1}{1-\lambda K}\cdot \|x\|.
\]
Since the function inside the square root is decreasing relative to $\lambda$, it attains its maximal value $4$ at $\lambda=\frac{2}{|3a+K|}$.  Therefore
\[
   \left\|  G_\lambda(x)  \right\|  \le\frac{3\|x\|}{1-\lambda K}\qquad \mbox{for all}\quad x\in\B_{\rho(\lambda)}.
\]

In the opposite case, we obtain similarly 
\[
   \left\| G_\lambda(x) \right\| \le\frac{\|x\|}{1+\lambda a}\qquad \mbox{for all}\quad x\in\B_{1+\lambda a}.
\]
Thus for all $x\in\B$,  $\left\|  G_\lambda(x)  \right\|  \le \alpha(\lambda) \to 0$ as $\lambda\to\infty$.  
    \end{proof}

\begin{corollary}\label{cor-1new}
  Let $f\in\N_a$ for some $a\ge0$ and $\alpha(\lambda)$ be defined as in Corollary~\ref{cor-unif}. If $ \alpha(\lambda) \le \|w\| <1$, then $\|w+\lambda f(w) \| \ge1$.
  \end{corollary}
\begin{proof}
  Indeed, suppose by contradiction that $x:=w+\lambda f(w)$ lies in~$\B$. Then by Definition~\ref{def-range_cond} and Proposition~\ref{lemma1}, $w=G_\lambda(x)$, and hence $\|w\| < \alpha(\lambda)$ by Theorem~\ref{thm-reso-general}.
\end{proof}

Further, by Definition~\ref{def-range_cond},
\[
 G_\lambda(x) + \lambda f(G_\lambda(x)) =x,
\]
hence
\[
\lambda \left\langle f(G_\lambda(x)) ,x^*\right\rangle = \|x\|^2 -  \left\langle G_\lambda(x),x^*\right\rangle .
\]
Thus estimate~\eqref{ineq-reso1} implies 
\[
 \lambda \Re \left\langle f(G_\lambda(x)) ,x^*\right\rangle \ge \|x\|^2 -  \alpha(\lambda)\cdot \|x\| ^2.
\]
So, denoting $a_\lambda=\frac1\lambda\left(1-\alpha(\lambda)\right)$, one immediately concludes:
\begin{corollary}\label{cor-generator}
 Let $f\in\N_a$ for some $a\ge0$ and $G_\lambda,\ \lambda>0, $ be the resolvent of~$f$. 
 Then  
 \[
 f\circ G_\lambda \in\N_{a_\lambda},\quad \mbox{ where  }\quad a_\lambda =\max\left(\frac{a}{1+\lambda a}\,, \frac{-2-\lambda K}{\lambda(1-\lambda K)} \right).
 \]
 \end{corollary}

 We proceed with a lower estimate and a covering result for non-linear resolvents.

  \begin{theorem}\label{thm-cover}
Under the above assumptions, the mapping $G_\lambda$ satisfies inequality
\[
\| G_\lambda(x) \| \ge \beta(\lambda)\cdot \|x\|,
\]
where $\beta(\lambda):=\displaystyle \frac{\alpha(\lambda)}{\alpha(\lambda)+\lambda M_f(\alpha(\lambda))}\,$, $\ \alpha(\lambda)$ is defined in Theorem~\ref{thm-reso-general} and the quantity $M_f(r)=\sup\limits_{\|y\|<r} \|f(y)\|$ is finite for any $r\in(0,1).$ 
Hence  in the case $X=\C^n$, the image $G_\lambda(\B)$ covers the ball of radius~$\beta(\lambda)$.
  \end{theorem}
  \begin{proof}
Assume by contradiction that $\| G_\lambda(x_0) \| < \beta(\lambda)\cdot \|x_0\|$ at some point $x_0\in\B$.

    It follows from \cite{Har} (see also \cite{R-Sbook, E-R-S-19}) that every holomorphically accretive mapping $f$ has unit radius of boundedness, so $M_f(r)$ is finite for each $r\in(0,1)$. Thus for all $y$ with $\|y\|\le r$ we have $\|f(y)\|\le\frac{M_f(r)\|y\|}{r}$ due to the Schwarz lemma, and then
    \[
    \left\| y +\lambda f(y)  \right\| \le \left( 1 + \frac{\lambda M_f(r)}{r}  \right)\|y\|.
    \]

We already know from Theorem~\ref{thm-reso-general} that $ \left\|  G_\lambda(x_0)  \right\|  \le \alpha(\lambda)\cdot \|x_0\|< \alpha(\lambda)$. Substitution $y=G_\lambda(x_0)$  and $r=\alpha(\lambda) $  in the last displayed formula leads to
    \begin{eqnarray*}
      \left\| G_\lambda(x_0) + \lambda f(G_\lambda(x_0))  \right\| &\le & \left( 1 + \frac{\lambda M_f( \alpha(\lambda))}{ \alpha(\lambda)}  \right)\|G_\lambda(x_0)\| \\
       &<& \left( 1 + \frac{\lambda M_f( \alpha(\lambda))}{ \alpha(\lambda)}  \right)\beta(\lambda)\cdot \|x_0\| =\|x_0\| ,
    \end{eqnarray*}
       which contradicts the definition of the resolvent. The contradiction implies that such point $x_0$ does~not~exist.
       
       The last assertion follows from the biholomorphicity of $G_\lambda$.
  \end{proof}

In the case where $X$ is a Hilbert space, so that $X^*$ can be naturally identified with $X$ and $J(x)=\{x\}$, the last theorem implies that the resolvents themselves are holomorphically accretive. 
\begin{corollary}\label{cor-G-gener}
Let  $f\in\N_a,\ a\ge0,$  and let $\{G_\lambda\}_{\lambda>0}$ be the resolvent family for $f$.  Then $G_\lambda\in\N_{b_\lambda}$ for every $\lambda>0$, where $b_\lambda=(1+\lambda a)  \beta^2(\lambda)$ and $\beta(\lambda)$ is defined in Theorem~\ref{thm-cover}.
\end{corollary}
\begin{proof}
   It follows from our assumptions that $G_\lambda$ satisfies the resolvent equation \eqref{G*}, that is,
\begin{equation}\label{resol}
 G_\lambda(x) + \lambda f(G_\lambda(x)) =x.
\end{equation}
Calculating the inner product with $G_\lambda(x)$ we get
\[
\| G_\lambda(x) \|^2 + \lambda \left\langle f(G_\lambda(x)),  G_\lambda(x) \right\rangle =\left\langle x, G_\lambda(x)\right\rangle.
\]
This relation together with \eqref{a-condi-1} imply 
\begin{eqnarray*}
  \Re \left\langle G_\lambda(x) , x \right\rangle&=& \| G_\lambda(x) \|^2 + \lambda \Re \left\langle f(G_\lambda(x)),  G_\lambda(x) \right\rangle \\
  & \ge & (1+\lambda a) \| G_\lambda(x) \|^2  \\
  & \ge &  (1+\lambda a)  \beta^2(\lambda)\cdot \|x\|^2,
\end{eqnarray*}
which completes the proof.
\end{proof}

\medskip


\section{Starlikeness  of resolvents}\label{sect-starlike}

\setcounter{equation}{0}

In this section we study the question whether a non-linear resolvent is a starlike (starlike of some order) mapping. Following \cite{HKL, GK2003} (see also \cite{STJ-77, E-R-S-04}), we recall the following definition.

\begin{definition}\label{def-star}
  Let  $h\in\Hol(\B,X),\ h(0)=0,$ be a biholomorphic mapping. It is called starlike if
  \begin{equation}\label{starl}
\Re  \left\langle   \left( h'(x) \right)^{-1} h(x), x^* \right\rangle\ge0
\end{equation}
for all $x\in\B$. Moreover, the mapping $h$ is called starlike of order $\gamma\in(0,1]$ if it satisfies
\begin{equation}\label{star-order}
  \left|   \frac{\left\langle   \left( h'(x) \right)^{-1} h(x), x^* \right\rangle}{\|x\|^2} - \frac{1}{2\gamma} \right| \le \frac{1}{2\gamma},\qquad x\in\B\setminus\{0\}.
\end{equation}
\end{definition}
 Starlike mappings of order $\gamma=0$ are just starlike, while the only mappings starlike of order $\gamma=1$ are linear ones.
Notice that the normalization $h'(0)=\Id$ is usually imposed. Since the relations~\eqref{starl}--\eqref{star-order} are invariant under the transformation $h\mapsto Bh$, where $B$ is an invertible linear operator, we conclude that the 
normalization is unnecessary.\vspace{2mm}

First we examine Conjecture~\ref{conje-1} above. While we cannot prove it completely, we provide its proof for values of the resolvent parameter $\lambda$ satisfying a special condition.  
Recall that by Theorem~\ref{thm-reso-general}, $\|G_\lambda(x)\|\le\alpha(\lambda),$ where $\alpha(\lambda)=\min\left(  \frac{3}{1-\lambda K} , \frac1{1+\lambda a} \right) \le1$, $K=K_A(\pi)<0$, ${A=f'(0)}.$ Also set $K_1=K_A(0)$. In the case $\alpha(\lambda)<1$  let us denote
\begin{eqnarray*}
  \Psi(\lambda) :=  \frac{2\lambda\left( K_1 -a  \right) }{1-\lambda K} \cdot  \sum\limits_{n=2}^\infty   n^{\frac{2n-1}{n-1}} (\alpha(\lambda) )^{n-1} ,\quad \lambda\ge0.
\end{eqnarray*}

\begin{lemma}\label{lem-aux1}
  Function $\Psi$  is decreasing 
  and satisfies $\lim\limits_{\lambda\to\infty} \Psi(\lambda)=0$.
\end{lemma}
This lemma can be verified directly. 

\begin{theorem}\label{thm-order1}
  Let $f\in\N_a$ and $\{G_\lambda\}_{\lambda>0}$ be the resolvent family for $f$. Denote by $\lambda^*$ the largest solution to the equation ${\Psi(\lambda)= 1}.$ Then for every $\lambda\ge\lambda^*$ the mapping $G_\lambda$ is starlike of order~$\frac12$.
\end{theorem}
\begin{proof}
Following Definition~\ref{def-star}, we have to show that
\[
\left| \frac{\left\langle   \left( G'(x) \right)^{-1} G(x), x^* \right\rangle}{\|x\|^2}-1\right| \le1\,,
\]
where $G=G_\lambda$, that is, $G$ satisfies~\eqref{resol}.
Differentiating \eqref{resol} we get $(\Id  + \lambda f'(G(x))G'(x)=\Id.$ Therefore
\begin{equation*}
\left( G'(x) \right)^{-1} G(x) =G(x) + \lambda f'(G(x))G(x).
\end{equation*}
According to this equality, we have show that
  \[
  \left| \left\langle G(x) + \lambda f'(G(x))G(x) -x, x^* \right\rangle \right|\le \|x\|^2.
  \]
  Denote $w=G(x)$ or, which is the same, $x=w + \lambda f(w)$. Hence it suffices to prove that 
 \begin{equation}\label{G-star-axi}
   \lambda\|f'(w)w-f(w)\|\le \|\lambda f(w) +w\|
 \end{equation} 
for $w\in G(\B)$. Represent $f$ by the series of homogeneous polynomials: $f(x)=Ax +\sum\limits_{n=2}^\infty P_n(x)$. Then
  \begin{eqnarray}\label{homo1}
     \|\lambda f(w)+w\| &=& \left\|w +\lambda Aw + \lambda\sum\limits_{n=2}^\infty P_n(w) \right\| \nonumber \\
     &\ge& \left\|(\Id +\lambda A)w\right\| -\left\|\lambda\sum\limits_{n=2}^\infty P_n(w) \right\| \nonumber \\
     &\ge& (1-\lambda K)\|w\| -  \lambda\sum\limits_{n=2}^\infty \left\|P_n(w) \right\|.
  \end{eqnarray}
  On the other hand,
   \begin{eqnarray}\label{homo2}
    \|f'(w)w-f(w)\|&=&   \left\| \sum\limits_{n=2}^\infty nP_n(w) -\sum\limits_{n=2}^\infty P_n(w) \right\| \nonumber \\
     &\le&    \sum\limits_{n=2}^\infty(n-1) \left\|P_n(w) \right\|.
  \end{eqnarray}
  The comparison \eqref{G-star-axi}--\eqref{homo2} shows that the result will follow when we show that $\lambda \sum\limits_{n=2}^\infty n \left\|P_n(w) \right\|\le (1-\lambda K)\|w\|$ or, which is the same,
  \[
  \frac{\lambda}{1-\lambda K} \cdot \sum\limits_{n=2}^\infty n \left\|P_n\right\| \|w\|^{n-1}\le 1.
  \]
 It follows from Proposition~\ref{propo-esim1} that $\left\| P_n \right\| \le 2 n^{\frac{n}{n-1}}\left( K_1 -a  \right)$. Also, it was proved in Corollary~\ref{cor-unif} that $\|w\|\le\alpha(\lambda).$ Therefore we have
  \begin{eqnarray*}
\frac{\lambda}{1-\lambda K} \!\!&\cdot&\!\!\sum\limits_{n=2}^\infty n \left\|P_n\right\| \|w\|^{n-1} \\
    &\le& \left. \frac{\lambda}{1-\lambda K} \sum\limits_{n=2}^\infty  2 n^{1+\frac{n}{n-1}}\left( K_1 -a  \right) t^{n-1} \right|_{t= \alpha(\lambda)}  =\Psi(\lambda).
  \end{eqnarray*}
  Thus the result follows.
\end{proof}


  Theorem~\ref{thm-order1}  relies on function $\Psi$ that can not be calculated explicitly.  It is possible to replace $\Psi$ by a 
  function that is larger than $\Psi$. For instance, taking in mind that $n^{\frac{2n-1}{n-1}}\le 2n^2$, one can replace $\Psi$ by
\begin{eqnarray*}
  \Psi_1(\lambda) &:=&  \frac{4\lambda\left( K_1 -a  \right) }{1-\lambda K} \cdot \frac{\alpha(\lambda)\left(4-3\alpha(\lambda) +(\alpha(\lambda) )^2 \right)} {(1-\alpha(\lambda) )^3}.
\end{eqnarray*}



Another issue arises when one asks on starlikeness {\it for small} values of $\lambda$. 
The next result is equivalent to \cite[Theorem~4.16 (ii)]{GHK2020}.

\begin{proposition}\label{propo-star-GHK}
 Let $f\in\N_0$ on the Euclidean ball $\B\subset\C^n$ and be normalized by $f'(0)=\Id$.  If $f'$ is bounded on $\B$ and hence ${\| f'(x) -\Id\|} <b,\ x\in\B,$ for some $b>\frac2{\sqrt{5}+2}$, then $G_\lambda$ is a starlike mapping for each $\lambda<\frac{2}{b(\sqrt{5}+2)-2}$.
 \end{proposition}

In our next result we refuse the  normalization in Proposition~\ref{propo-star-GHK} and evaluate order of starlikeness. To formulate it, let us denote
\[
\gamma(t)= \left\{ \begin{array}{ll}
                    \frac{2(1-2t)}{2-t}, & 0<t<\frac25,\vspace{2mm}\\
                    \frac{4-8t -t^2}{2(4-8t+3t^2)},  & \frac25\le t \le\frac2{2+\sqrt5}.
                  \end{array}   \right.  
\]

\begin{theorem}\label{thm-order2} 
  Let $f\in\N_a$  on the Euclidean ball $\B\subset\C^n$, $A=f'(0)$ and $\{G_\lambda\}_{\lambda>0}$ be the resolvent family for $f$.  Let $f'$ be bounded on $\B$ and hence $\| f'(x) -A \| <b,\ x\in\B,$ for some $b$. Let $\lambda$ be such that $d(\lambda)\le \frac{2}{b(2+\sqrt{5})}$, where $d(\lambda) :=  \frac{\lambda\alpha(\lambda)}{1-\lambda K}$. Then the mapping $G_\lambda$ is starlike of order $\gamma(bd(\lambda))$.
  \end{theorem}

\begin{proof}
We already mentioned that starlikeness property is invariant under the transformation $h\mapsto Bh$, where $B$ is an invertible linear operator. Hence we can consider the normalized resolvent $\widetilde{G}(x)=(\Id +\lambda A)G(x)$, where $G=G_\lambda$ for some $\lambda$ that satisfies $d(\lambda)\le \frac{2}{b(2+\sqrt{5})}$. Then (as above we write $w$ instead of $G(x$))
 \begin{eqnarray*}
    \widetilde{G}'(x)&=& (\Id +\lambda A) (\Id +\lambda f'(w))^{-1} \\
    &=& \left[ (\Id +\lambda f'(w)) (\Id +\lambda A)^{-1} \right]^{-1} \\
    &=& \left[ \Id + \lambda (f'(w)-A) (\Id +\lambda A)^{-1} \right]^{-1} \\
    &=& \Id+ \sum_{k=1}^\infty (-\lambda)^k \left[ (f'(w)-A) (\Id +\lambda A)^{-1} \right]^k.
 \end{eqnarray*}
 Therefore
  \begin{eqnarray*}
   \left\| \widetilde{G}'(x) -\Id\right\| &\le & \sum_{k=1}^\infty \lambda^k \left\| (f'(w)-A) (\Id +\lambda A)^{-1} \right\|^k       \\
    &\le & \sum_{k=1}^\infty \lambda^k \left\| f'(w)-A\right\|^k \left\|(\Id +\lambda A)^{-1} \right\|^k \\
     &\le & \sum_{k=1}^\infty \lambda^k \left\| f'(w)-A\right\|^k \frac1{(1-\lambda K)^k} .
 \end{eqnarray*}
 To proceed we note that the mapping $f'-A$ is an element of $\Hol(\B,L(X))$ and vanishes at the origin. Therefore by the Schwarz lemma $\left\| f'(w)-A\right\|\le b\|w\|. $ In addition, $\|w\|\le \alpha(\lambda)$ by Theorem~\ref{thm-reso-general}. Thus 
 \begin{eqnarray*}
   \left\| \widetilde{G}'(x) -\Id\right\| &\le & \sum_{k=1}^\infty \lambda^k b^k\alpha^k(\lambda) \frac1{(1-\lambda K)^k} \\
   &=&  \sum_{k=1}^\infty \left( b d(\lambda)  \right)^k = \frac{bd(\lambda)}{1 - bd(\lambda)}   .
 \end{eqnarray*}
 
 Since $bd(\lambda)<\frac2{\sqrt{5}+2}<1 $, the summing in the above series is correct. Moreover, the value $t=bd(\lambda)$ belongs to the domain of the function $\gamma(t). $ Since the last function is the inverse to the function $2N_1(\cdot)$ defined in  \cite[Corollary 2]{L-Z}, the result follows from this corollary.
 \end{proof}
 
 \begin{corollary}
   Let $f\in\N_0$  on the Euclidean ball $\B\subset\C^n,$ $f'(0)=\Id$ and $\{G_\lambda\}_{\lambda>0}$ be the resolvent family for $f$.  Let $\| f'(x) -\Id \| <b,\ x\in\B,$ for some $b$. If either $\lambda< 1+\frac{21b}{4}-\sqrt{\left(\frac{21b}{4}\right)^2+\frac{21b}{2}}$, or $\lambda >1+\frac{21b}{4} + \sqrt{\left(\frac{21b}{4}\right)^2+\frac{21b}{2}} $, then the mapping $G_\lambda$ is starlike of order $\frac12$.
 \end{corollary}

From the point of view of Theorem~\ref{thm-order2}, it is relevant to evaluate $\max d(\lambda)$. The standard calculus shows that if $9a\ge |K|$ then the maximum is attained at $\lambda=\frac1{\sqrt{a|K|}}$, and otherwise at $\lambda=\frac{2}{|3a+K|}\,.$ Calculating the maximum values  in both cases we conclude:

\begin{corollary}\label{cor-order1}
(a)  If $|K|\le 9a$ and $\|f'(x)-A\|\le \frac{2(\sqrt{a}+\sqrt{|K|})^2}{2+\sqrt5}$, then for any $\lambda>0$  the mapping $G_\lambda$ is starlike of order $\gamma\left(\frac{2(\sqrt{a}+\sqrt{|K|})^2}{2+\sqrt5} d(\lambda)  \right)$.

(b) If $9a<|K|$ and $\|f'(x)-A\|\le \frac{3(|K|-a)^2}{(2+\sqrt5)(|K|-3a)}$, then for any $\lambda>0$ the mapping $G_\lambda$ is starlike of order $\gamma\left(\frac{3(|K|-a)^2}{(2+\sqrt5)(|K|-3a)} d(\lambda)  \right)$.
\end{corollary}

\end{document}